\documentclass[a4paper,10pt,leqno]{amsart}
\usepackage[total={6in,9in},
top=1in, left=1in, right=1in, bottom=1in]{geometry}
\usepackage{amsmath}
\usepackage[italian, english]{babel}
\usepackage{amsfonts}
\usepackage{amssymb}
\usepackage{dsfont}
\usepackage{mathrsfs}
\usepackage{graphicx}
\usepackage{setspace}
\usepackage{fancyhdr}
\usepackage{amsthm}
\usepackage{empheq}
\usepackage{cases}
\usepackage[all]{xy}
\usepackage{stmaryrd}
\usepackage{mathrsfs}
\usepackage[colorlinks, citecolor=blue, linkcolor=red]{hyperref}

\newcommand{\erre}{\mathds{R}}

\newcommand{\ricc}{\operatorname{Ric}}

\newcommand{\hess}{\nabla^2}

\newcommand{\ra}{\rightarrow}

\newcommand{\set}[1]{{\left\{#1\right\}}}               % tra parentesi graffe
\newcommand{\pa}[1]{{\left(#1\right)}}                  % tra tonde
\newcommand{\sq}[1]{{\left[#1\right]}}                  % tra quadre
\newcommand{\abs}[1]{{\left|#1\right|}}                 % valore assoluto
\newcommand{\pair}[1]{\left\langle#1\right\rangle}      % pairing

\newcommand{\eps}{\varepsilon}                           % epsilon
\renewcommand{\hat}[1]{\widehat{#1}}

\newcommand{\We}{\hat{E}}

\newtheorem{theorem}{\textbf{Theorem}}[section]
\newtheorem{ackn}{Acknowledgments\!}
\newtheorem{lemma}[theorem]{\textbf{Lemma}}
\newtheorem{proposition}[theorem]{\textbf{Proposition}}
\newtheorem{cor}[theorem]{\textbf{Corollary}}

\theoremstyle{remark}
\newtheorem{rem}[theorem]{\textbf{Remark}}

\numberwithin{equation}{section}

\title[Steady and expanding Ricci solitons]
{Classification of expanding and steady Ricci solitons \\with integral curvature decay}

\date{\today} 



\begin{document}

\maketitle

\begin{center}
\textsc{\textmd{G. Catino\footnote{Dipartimento di Matematica, Politecnico di Milano, Piazza Leonardo da Vinci 32, Milano, Italy, 20133. Email: giovanni.catino@polimi.it. Supported
by GNAMPA project ``Equazioni di evoluzione geometriche e strutture di tipo Einstein''.}, P.
Mastrolia\footnote{Dipartimento di Matematica, Universit\`{a} degli Studi di Milano, Via Saldini 50, Milano, Italy, 20133.
Email: paolo.mastrolia@gmail.com. Supported
by GNAMPA project ``Analisi Globale ed Operatori Degeneri''.} and D. D. Monticelli\footnote{Dipartimento di Matematica, Universit\`{a} degli Studi di Milano, Via Saldini 50, Milano, Italy, 20133.
Email: dario.monticelli@gmail.com. Supported
by GNAMPA project ``Analisi Globale ed Operatori Degeneri''. \\ The  authors are members of the Gruppo Nazionale per
l'Analisi Matematica, la Probabilit\`{a} e le loro Applicazioni (GNAMPA)
of the Istituto Nazionale di Alta Matematica (INdAM).
} }}
\end{center}
\begin{abstract}
In this paper we prove  new classification results for nonnegatively curved gradient expanding and steady Ricci solitons in dimension three and above, under suitable integral assumptions on the scalar curvature of the underlying Riemannian manifold.  In particular we show that the only complete expanding solitons with nonnegative sectional curvature and integrable scalar curvature are quotients of the Gaussian soliton, while in the steady case we prove rigidity results under sharp integral scalar curvature decay. As a corollary, we obtain that the only three dimensional steady solitons  with  less than quadratic volume growth are quotients of $\erre\times\Sigma^{2}$, where $\Sigma^{2}$ is Hamilton's cigar.
\end{abstract}

\section{Introduction and main results}

A gradient Ricci soliton is a smooth $n$--dimensional, connected,
Riemannian manifold $M^n$ with metric $g$ satisfying
\begin{equation}\label{2}
\ricc+\nabla^2f=\lambda g
\end{equation}
for some smooth potential function $f:M^n\ra \erre$ and a real constant  $\lambda$. The soliton is called
 expanding, steady or shrinking if, respectively, $\lambda<0$, $\lambda=0$ or $\lambda>0$. When $f$ is constant, a gradient Ricci soliton is an Einstein manifold. Ricci solitons generate self-similar solutions to the Ricci flow and often arise as singularity models of the flow; therefore, it is crucial to study and classify them in order to
understand the geometry of singularities.

The two dimensional case is well understood and all complete Ricci solitons have been classified, see for instance the very recent \cite{Bernstein} and references therein. In particular, it is well known that the only gradient steady Ricci soliton with positive curvature is  Hamilton's cigar $\Sigma^{2}$, see  \cite{hamilton}.

In dimension three, due to the efforts of Ivey \cite{Ivey}, Perelman \cite{perelman1}, Ni and Wallach \cite{NiWallach} and  Cao, Chen and Zhu \cite{CCZ}, shrinking solitons have been completely classified: they are quotients of either the round sphere $\mathds{S}^{3}$, the round cylinder $\mathds{R}\times\mathds{S}^{2}$ or the shrinking Gaussian soliton $\mathds{R}^{3}$.

In the steady three dimensional case the known examples are given by quotients of $\mathds{R}^{3}$, $\mathds{R}\times\Sigma^{2}$ and the rotationally symmetric one constructed by Bryant \cite{bryant}. In the seminal paper by Brendle \cite{BrendleRot}, it was shown that Bryant soliton is the only nonflat, $k$-noncollapsed, steady soliton, proving a famous conjecture by Perelman \cite{perelman1}. It is still an open problem to classify three dimensional steady solitons which do not satisfy the $k$-noncollapsing condition; see Cao \cite{CaoHe} for an interesting result in this direction.

The case of expanding solitons is far less rigid; however, some properties and classification theorems have been proved in the recent years by various authors, see for instance \cite{PetWyPac}, \cite{ma}, \cite{PRSMathZ},  \cite{chenderuelle}, \cite{SchulzeSimon}, \cite{Chodosh}, \cite{deruelle2} and references therein.

The aim of this paper is to prove new classification results of gradient expanding and steady solitons in dimension three and above under integral assumptions on the scalar curvature. Note that  similar conditions have been considered by Deruelle \cite{deruelle} in the steady case (although the author exploits a completely different approach).

 In particular we prove the following

\begin{theorem}\label{TH_expanding}
  Let $\pa{M^n, g}$ be a complete gradient expanding Ricci soliton of dimension $n\geq 3$ with nonnegative sectional curvature. If $R \in L^1(M^{n})$, then $M^n$ is isometric to a quotient of the Gaussian soliton $\erre^n$.
\end{theorem}

\begin{theorem}\label{TH_steady}
    Let $\pa{M^n, g}$ be a complete gradient steady Ricci soliton of dimension $n\geq 3$ with nonnegative sectional curvature. Suppose that
\[
 \liminf_{r \ra +\infty } \,\, \frac{1}{r}\int_{B_{r}(o)}R = 0 .
\]
Then, $M^n$ is isometric to a quotient of $\erre^n$ or  $\erre^{n-2}\times\Sigma^{2}$, where $\Sigma^{2}$ is the cigar soliton.
\end{theorem}

In the three dimensional case we can prove the analogous results under weaker assumptions.

\begin{theorem}
  \label{TH_expanding3dim}
  Let $\pa{M^3, g}$ be a three dimensional complete gradient expanding Ricci soliton with  nonnegative Ricci curvature. If $R \in L^1(M^{3})$, then $M^3$ is isometric to a quotient of the Gaussian soliton $\erre^3$.
\end{theorem}

\begin{theorem}
  \label{TH_steady3dim}
    Let $\pa{M^3, g}$ be  a three dimensional complete gradient steady Ricci soliton. Suppose that
\[
 \liminf_{r \ra +\infty } \,\, \frac{1}{r}\int_{B_{r}(o)}R = 0 .
\]
Then $M^3$ is isometric to a quotient of $\erre^3$ or  $\erre\times\Sigma^{2}$, where $\Sigma^{2}$ is the cigar soliton.
\end{theorem}

\begin{rem}
As it will be clear from the proofs of Theorems \ref{TH_expanding} and \ref{TH_expanding3dim}, instead of $R\in L^{1}(M^{n})$ we can assume that
\[
  \liminf_{r \ra +\infty } \,\, \int_{B_{2r}(o)\setminus B_{r}(o)}R = 0.
\]
\end{rem}

\begin{rem}
The quantity
\begin{equation}\label{liminf}
 \liminf_{r \ra +\infty } \,\, \frac{1}{r}\int_{B_{r}(o)}R
\end{equation}
that appears in Theorems \ref{TH_steady} and \ref{TH_steady3dim}  is independent of the choice of the center $o\in M^{n}$. Moreover, note that our assumptions in the steady case do not imply \emph{a priori} that the scalar curvature goes to zero at infinity, in contrast with the results in \cite{deruelle}. In fact, in \cite{deruelle} it is assumed that $R\in L^1(M^n)$. This, using the hypothesis that the steady Ricci soliton has nonnegative sectional curvature, implies that the scalar curvature is nonnegative, bounded, and globally Lipschitz, and thus that $R\ra 0$ at infinity.
\end{rem}

As a consequence of the integral decay estimate in \cite{deruelle} (see Lemma \ref{lemma4.3deruelle}), it follows that the assumption in Theorems \ref{TH_steady} and \ref{TH_steady3dim} holds if $g$ has less than quadratic volume growth, i.e. $\operatorname{Vol}\pa{B_r(o)}=o(r^2)$ as $r\ra+\infty$. This immediately implies the following
\begin{cor}\label{corquadratic}
The only complete gradient steady Ricci solitons of dimension $n\geq 3$ with nonnegative sectional curvature and  less than quadratic volume growth are quotients of $\erre^{n-2}\times\Sigma^{2}$.
\end{cor}
In particular, in dimension three the nonnegativity assumption on the curvature is automatically satisfied (see \cite{chen}), implying
\begin{cor}\label{corquadratic3}
The only three dimensional complete gradient steady Ricci solitons  with  less than quadratic volume growth are quotients of $\erre\times\Sigma^{2}$.
\end{cor}

We note that the condition in Theorem \ref{TH_expanding3dim} is sharp: in fact there exists a  rotationally symmetric example constructed by Bryant in \cite{bryant} (see also the appendix in \cite{Chodosh}) which has positive sectional curvature, quadratic curvature decay at infinity and Euclidean volume growth.
%Also, the nonnegative sectional curvature assumption cannot be dropped since the example discovered by Feldman, Ilmanen and Knopf \cite{****} has exponential ************** CONTROLLARE.
Moreover, the steady Bryant soliton has positive sectional curvature, linear curvature decay and quadratic volume growth; this shows that the assumptions in Theorem \ref{TH_steady3dim} and Corollary \ref{corquadratic3} are sharp as well.

We explicitly remark that Theorems \ref{TH_steady} and \ref{TH_steady3dim} improve a result in \cite{deruelle}, while the results in the expanding case, to the best of our knowledge, are completely new and should be compared with \cite[Theorem 4]{PRSMathZ} and \cite[Theorem 4.5]{PetWyPac}, where the required integral conditions involve the measure $e^{-f} d\mu$, and the weight $e^{-f}$, under mild assumptions on the curvature, has exponential growth (see e.g. \cite[Lemma 5.5]{CaoCatAltri}). We also note that, in dimension three, the condition
\begin{equation}\label{pippo}
 \liminf_{r \ra +\infty } \,\, \frac{1}{r}\int_{B_{r}(o)}R \geq k > 0
\end{equation}
implies the $k$-noncollapsing of balls with sufficiently large radii,\emph{ a priori} nonuniformly with respect to the center. It would be extremely interesting either to show that the only three dimensional gradient steady Ricci soliton satisfying \eqref{pippo} is, up to scaling, the Bryant soliton, or to construct a ($k$-collapsed) counterexample. The first case, together with Theorem \ref{TH_steady3dim}, would complete the classification of steady solitons in dimension three.

One of the main tool in our analysis is a geometric $(0, 2)$-tensor that we call the \emph{weighted Einstein tensor} $\hat{E}$, and which is defined as

\begin{equation}\label{def_T}
  \hat{E}=\pa{\ricc - \tfrac{1}{2}R \, g} e^{-f},
\end{equation}
where $f$ is the soliton potential. The weighted Einstein tensor  appeared for the first time in \cite{codazzit}, where the authors observed that $\We$ is a Codazzi tensor on every gradient three dimensional Ricci soliton. Here we prove, in Section \ref{SectW}, that on every gradient Ricci soliton $\We$ satisfies the  Weitzenb\"ock formula
\begin{equation}\label{Anita}
 \frac{1}{2}\Delta|\We|^2 = |\nabla\We|^2 - \frac{1}{2}\langle\nabla|\We|^2, \nabla f\rangle-\pa{n-2}\lambda|\We|^2  + Q,
\end{equation}
where $Q$ is a cubic curvature term. Quite surprisingly, we are able to show that this quantity satisfies nice algebraic properties (see Section \ref{SectW} and the Appendix) under suitable curvature assumptions;  namely, we prove that $Q\geq0$ if the sectional curvature (or the Ricci curvature, in dimension three) is nonnegative, and we completely  characterize the equality case. We highlight the fact that equation \eqref{Anita} is only effective when $\lambda\leq 0$; this feature allows us to exploit, in the expanding and steady case, a technique  reminiscent of those used to prove earlier results concerning  gradient shrinking Ricci solitons (see e.g. \cite{NiWallach}, \cite{PetWy}, \cite{Catrino}, \cite{MuWa}, \cite{WWW}).

Exploiting the above formula and using a careful integration by parts argument, in Section \ref{SecExp} we prove Theorems \ref{TH_expanding} and \ref{TH_expanding3dim} in the expanding case, while in Section \ref{SecSteady} we prove Theorems \ref{TH_steady} and \ref{TH_steady3dim} in the steady case.

\section{A Weitzenb\"ock formula for the weighted Einstein tensor}\label{SectW}

Let $(M^{n},g)$ be a complete gradient Ricci soliton of dimension $n\geq 3$, that is a Riemannian manifold satisfying the equation
\begin{equation}\label{sol}
\ricc + \hess f \,=\, \lambda \,g
\end{equation}
for some smooth function $f:M\ra\erre$ and some constant $\lambda\in\erre$.
For the weighted Einstein tensor $\We$ defined in \eqref{def_T} we prove the following
\begin{proposition}
Let $(M^{n},g)$ be a complete gradient Ricci soliton of dimension $n\geq 3$. Then
 %   \begin{equation}\label{eq_laplTsquared_RiemRT_mdim}
%    \frac{1}{2}\Delta\abs{\We}^2 = \abs{\nabla\We}^2 - \frac{1}{2}\pair{\nabla\abs{\We}^2, \nabla f}-\pa{m-2}\lambda\abs{\We}^2  -2\Riem\pa{\We, \We}+\frac{m-2}{2}Re^{-2f}\sq{\frac{1}{2}R^2-\abs{\ricc}^2}.
%  \end{equation}
   \begin{equation}\label{eq_laplTsquared_RiemRT_mdimconE}
    \frac{1}{2}\Delta|\We|^2 = |\nabla\We|^2 - \frac{1}{2}\langle\nabla|\We|^2, \nabla f\rangle-\pa{n-2}\lambda|\We|^2  -2\operatorname{Rm}(\We, \We)-\frac{n-2}{2}R\sq{|\We|^2-\frac{1}{n-2}\pa{\operatorname{tr}\pa{\We}}^2},
  \end{equation}
  where $\operatorname{Rm}$ is the Riemann curvature tensor, $\operatorname{Rm}\pa{\We, \We}=R_{ijkl}\We_{ik}\We_{jl}$ and $\operatorname{tr}$ is the trace.
\end{proposition}
\begin{proof}
  From \eqref{def_T} we have, on a local orthonormal frame,
\begin{equation*}
  e^{f}\We_{ij} = R_{ij}-\frac{1}{2}R\delta_{ij},
\end{equation*}
which implies, taking the covariant derivative,
\begin{equation}\label{eq_covderivTmdim}
  e^{f}\pa{f_k\We_{ij}+\We_{ij, k}} = R_{ij, k} - \frac12 R_k\delta_{ij}.
\end{equation}

 Taking the divergence of the previous  equation  we get
  \begin{equation}\label{eq_divt_ij}
    e^f\pa{\abs{\nabla f}^2\We_{ij}+2f_k\We_{ij, k}+\Delta f \We_{ij}+\We_{ij, kk}} = R_{ij, kk}-\frac12\Delta R\delta_{ij}.
  \end{equation}
  Now we recall that, for a gradient Ricci soliton, we have the validity of the following equations (see e.g. \cite{ELNM} or \cite{MRR_GS}):
  \[
  R_{ij, kk} = f_kR_{ij, k} + 2\lambda R_{ij}-2R_{kt}R_{ikjt},
  \]
  \[
  \Delta R = \pair{\nabla f, \nabla R} + 2\lambda R - 2\abs{\ricc}^2,
  \]
  \[
  R+\Delta f = n\lambda.
  \]
  Inserting the previous relations in equation \eqref{eq_divt_ij}, using the definition of $\We$ (which implies that $R_{ij}=e^f\We_{ij}+\frac12 R\delta_{ij}$ and $\abs{\ricc}^2=e^{2f}|\We|^2-\frac{(n-4)}{4} R^2$) and simplifying  we deduce
  \begin{equation}\label{eq_doublecovderti}
    e^f\pa{\We_{ij, kk} + f_k\We_{ij, k}+(n-2)\lambda \We_{ij}}=-2e^f\We_{kt}R_{ikjt} -\frac{(n-2)}{4}R^2\delta_{ij} + e^{2f}|\We|^2 \delta_{ij}.
  \end{equation}
  Now we contract \eqref{eq_doublecovderti} with $\We_{ij}$, observing that $\operatorname{tr}(\We) = \We_{tt} = -\frac{(n-2)}{2} R e^{-f}$, and we obtain
  \begin{equation}
    \We_{ij}\We_{ij, kk} = -\frac12\langle\nabla|\We|^2, \nabla f\rangle -(n-2)\lambda|\We|^2 - 2\operatorname{Rm}(\We, \We) - \frac{(n-2)}{2} R\sq{|\We|^2-\frac{1}{n-2}\pa{\operatorname{tr}(\We)}^2},
  \end{equation}
  which easily implies \eqref{eq_laplTsquared_RiemRT_mdimconE} since  $\frac12\Delta|\We|^2 = |\nabla\We|^2 +\We_{ij}\We_{ij, kk}$.
%  The equivalence with the other four version is just a simple computation exploiting the fact that, in the three dimensional case, the Riemann curvature tensor decomposes as
%  \[
%  R_{ikjt} = R_{ij}\delta_{kt}-R_{it}\delta_{kj}+R_{kt}\delta_{ij}-R_{kj}\delta_{it}-\frac{R}{2}\pa{\delta_{ij}\delta_{kt}-\delta_{it}\delta_{kj}}.
%  \]
\end{proof}

\begin{cor}\label{COR_ndimT}
Let $(M^{n},g)$ be a complete gradient Ricci soliton of dimension $n\geq 3$. Then

  \begin{equation}\label{eq_laplTsquared_RiemRwe_mdimT}
    \frac{1}{2}\Delta|\We|^2 = |\nabla\We|^2 - \frac{1}{2}\langle\nabla|\We|^2, \nabla f\rangle-\pa{n-2}\lambda|\We|^2  + Q
  \end{equation}
  where $$Q:=e^{-2f}\sq{\frac{(n-2)^3}{4n^2}R^3-2R_{ikjt}T_{ij}T_{kt}-\frac{\pa{n-2}\pa{n-4}}{2n}R\abs{T}^2},$$
  where $T$ is the traceless Ricci tensor.
\end{cor}
\begin{proof} We recall that, in a local orthonormal frame, the components $T_{ij}$ of $T$ are
\[
T_{ij}=R_{ij}-\frac{R}{n}\delta_{ij}.
\]
Using the definition of $\We$  we deduce that
\[
-2R_{ikjt}\We_{ij}\We_{kt} e^{2f} = -2R_{ikjt}T_{ij}T_{kt} + \frac{2(n-2)}{n}R\abs{T}^2-\frac{(n-2)^2}{2n^2}R^3 .
\]
Now the claim follows inserting the previous equation into \eqref{eq_laplTsquared_RiemRT_mdimconE}.
\end{proof}

In the particular case of dimension three we have
\begin{cor}\label{COR_3dimT}
Let $(M^{3},g)$ be a three dimensional complete gradient Ricci soliton. Then
   \begin{equation}\label{eq_laplTsquared_2RiccScalcoro}
    \frac{1}{2}\Delta|\We|^2 = |\nabla\We|^2 - \frac{1}{2}\langle\nabla|\We|^2, \nabla f\rangle-\lambda|\We|^2 + Q,
  \end{equation}
  with $$Q=e^{-2f}\sq{4R_{ij}R_{jk}R_{ki}-\frac{7}{2}R\abs{\ricc}^2+\frac{3}{4}R^3}.$$
\end{cor}
\begin{proof}
  The proof of the corollary is a simple computation using the fact that, in dimension three, one has
  \[
  \abs{T}^2 = \abs{\ricc}^2 - \frac13 R^2,
  \]
  \[
  R_{ijkt}T_{ij}T_{kt} = R_{ijkt}R_{ij}R_{kt} -\frac{2}{3}R\abs{\ricc}^2 +\frac{1}{9}R^3
  \]
  and
  \[
  R_{ijkt} = R_{ik}\delta_{jt}-R_{it}\delta_{jk}+R_{jt}\delta_{ik}-R_{jk}\delta_{it} -\frac{R}{2}\pa{\delta_{ik}\delta_{jt}-\delta_{it}\delta_{jk}}.
  \]
 
\end{proof}

In the next proposition we prove the main integral estimate that will be used in the proof of our results.
\begin{proposition} \label{prop_intmain}
Let $(M^{n},g)$ be a complete gradient Ricci soliton of dimension $n\geq 3$. Then, either $(M^{n},g)$ is Ricci flat or, for every nonnegative cutoff function $\varphi$ with compact support in $M^n$, we have

\begin{equation}\label{eq_intmain}
\int_{M} \frac{\sq{Q-(n-2)\lambda|\We|^{2}} \varphi^3 e^{f}}{|\We|}  \leq - 3\int_{M} \langle \nabla |\We|, \nabla \varphi \rangle\varphi^2 e^{f}
\end{equation}

\end{proposition}
\begin{proof}
For every $\eps\geq0$ define $$\Omega_\eps:=\{x\in M^n\,|\, |\We(x)|\geq\eps\}$$ and let
$$h_\eps(x):=\begin{cases}|\We(x)|\qquad\textrm{ if }x\in\Omega_\eps,\\
\eps\qquad\qquad\textrm{ if }x\in M\setminus\Omega_\eps.\end{cases}$$
Let $\varphi$ be a smooth nonnegative cutoff function with compact support in $M$. We  multiply equation \eqref{eq_laplTsquared_RiemRwe_mdimT} by $h_{\eps}^{-1}\,\varphi^3 \, e^{f}$ and integrate on $M^n$, deducing

\begin{align*}
 \frac{1}{2}\int_{M} \frac{\Delta |\We|^{2} \varphi^3 e^{f}}{h_{\eps}} &= \int_{M} \frac{\langle \nabla |\We|, \nabla h_{\eps}\rangle |\We|\varphi^3 e^{f} }{h_{\eps}^{2}} -  3\int_{M} \frac{\langle \nabla |\We|, \nabla \varphi \rangle |\We|\varphi^2 e^{f} }{h_{\eps}}- \int_{M} \frac{\langle \nabla |\We|, \nabla f \rangle \varphi^3 e^{f} }{h_{\eps}}.
\end{align*}
Since $h_{\eps}=|\We|$ on $\Omega_{\eps}$ and $\nabla h_{\eps}=0$ on $M\setminus \Omega_{\eps}$, we obtain

\begin{align*}
 \frac{1}{2}\int_{M} \frac{\Delta |\We|^{2} \varphi^3 e^{f}}{h_{\eps}} &= \int_{M} \frac{|\nabla h_{\eps}|^{2} \varphi^3 e^{f} }{h_{\eps}} -  3\int_{M} \frac{\langle \nabla |\We|, \nabla \varphi \rangle |\We|\varphi^2 e^{f} }{h_{\eps}}- \int_{M} \frac{\langle \nabla |\We|, \nabla f \rangle \varphi^3 e^{f} }{h_{\eps}}.
\end{align*}
Equation \eqref{eq_laplTsquared_RiemRwe_mdimT} and Kato's inequality yield

\begin{align*}
0 &= \int_{M}\frac{|\nabla \We|^{2} \varphi^3 e^{f}}{h_{\eps}} - \int_{M} \frac{|\nabla h_{\eps}|^{2} \varphi^3 e^{f} }{h_{\eps}} + 3\int_{M} \frac{\langle \nabla |\We|, \nabla \varphi \rangle |\We|\varphi^2 e^{f} }{h_{\eps}} + \int_{M} \frac{\sq{Q-(n-2)\lambda|\We|^{2}} \varphi^3 e^{f}}{h_{\eps}} \\
&\geq \int_{M}\frac{|\nabla |\We||^{2} \varphi^3 e^{f}}{h_{\eps}} - \int_{M} \frac{|\nabla h_{\eps}|^{2} \varphi^3 e^{f} }{h_{\eps}} + 3\int_{M} \frac{\langle \nabla |\We|, \nabla \varphi \rangle |\We|\varphi^2 e^{f} }{h_{\eps}} + \int_{M} \frac{\sq{Q-(n-2)\lambda|\We|^{2}} \varphi^3 e^{f}}{h_{\eps}} \\
&= \int_{M\setminus \Omega_{\eps}}\frac{|\nabla \We|^{2} \varphi^3 e^{f}}{h_{\eps}} + 3\int_{M} \frac{\langle \nabla |\We|, \nabla \varphi \rangle |\We|\varphi^2 e^{f} }{h_{\eps}} + \int_{M} \frac{\sq{Q-(n-2)\lambda|\We|^{2}} \varphi^3 e^{f}}{h_{\eps}} \\
&\geq 3\int_{M} \frac{\langle \nabla |\We|, \nabla \varphi \rangle |\We|\varphi^2 e^{f} }{h_{\eps}} + \int_{M} \frac{\sq{Q-(n-2)\lambda|\We|^{2}} \varphi^3 e^{f}}{h_{\eps}}.
\end{align*}
Now, since every complete Ricci soliton is real analytic in suitable coordinates (see \cite{bando} and \cite[Theorem 2.4]{catRig}), by the unique continuation property one has that either $|\We|\equiv 0 $ or the zero set of $|\We|$ has zero measure. In the first case, Bianchi identity implies that $g$ is Ricci flat while in the second case, taking the limit as $\eps\rightarrow 0$, since $ |\We| h_{\eps}^{-1} \rightarrow 1$ almost everywhere on $M^n$, inequality \eqref{eq_intmain} follows.
\end{proof}

\section{Expanding case: proof of Theorems \ref{TH_expanding} and \ref{TH_expanding3dim}}\label{SecExp}

\subsection{The $n$-dimensional case}

Let $\pa{M^n, g}$ be a complete gradient expanding Ricci soliton of dimension $n\geq 3$ with nonnegative sectional curvature and assume that $R\in L^1(M^n)$.
%  \begin{equation}\label{EQ_expannuli2}
%    \int_{B_{2r}\setminus B_r}R \ra 0 \quad \text{ as }\, r \ra +\infty.
%    \end{equation}
From Proposition \ref{prop_intmain}, we have that either the soliton is Ricci flat (hence flat, since $g$ has nonnegative sectional curvature) or
\begin{equation}\label{giani}
\int_{M} \frac{\sq{Q-(m-2)\lambda|\We|^{2}} \varphi^3 e^{f}}{|\We|}  \leq - 3\int_{M} \langle \nabla |\We|, \nabla \varphi \rangle\varphi^2 e^{f}
\end{equation}
for every nonnegative smooth cutoff function $\varphi$ with compact support in $M^n$. Since $(M^n,g)$ has nonnegative sectional curvature,  $|\operatorname{Rm}|\leq \alpha R$, for some positive constant $\alpha$ (see e.g. \cite{BourgK}) Moreover, for expanding solitons with nonnegative Ricci curvature we have that the scalar curvature $R$ is bounded, see for instance \cite[Proposition 2.4]{deruelle2}. Thus $g$ has bounded curvature. We recall that for every gradient Ricci soliton we have  Hamilton's identity
\[
R+|\nabla f|^{2} -2\lambda f = c
\]
for some real constant $c$ (\cite{hamilton2}). Since $R\geq 0$, we deduce that $|\nabla f|^{2} \leq 2\lambda f+c$. By \cite[Lemma 5.5]{CaoCatAltri} there exist positive constants $c_{1},c_{2},c_{3}$ such that
\begin{equation}\label{stimaeffe}
-\frac{\lambda}{2}(r(x)-c_{1})^{2}-c_{2}\leq -f(x) \leq -\frac{\lambda}{2}(r(x)+c_{3})^{2},
\end{equation}
where $r(x)=dist(x,o)$ for some fixed origin $o\in M^n$; in particular $f$ is proper. We define, for $t \gg 1$,
\[
\Omega_t = \set{x \in M^n : -f(x)\leq t}.
\]
We choose $\varphi(x)=\psi(f(x))$, where $\psi : [0, +\infty) \ra [0, +\infty)$ is a smooth nonincreasing function with support in $[0, 2t]$, such that $\psi\equiv 1$ in $[0, t]$ and
    \[
    \abs{\psi'(s)} \leq \frac{c}{s}\psi^{3/4}(s)\qquad \text{ and }\, \abs{\psi''(s)} \leq \frac{c}{s^2}\psi^{1/2}(s)
    \]
    for some $c>0$. Since in $\Omega_t$ we have $\abs{\nabla f}\leq c\sqrt{f}\leq c\sqrt{t}$,  we deduce
   \begin{equation}\label{phi1}
      \abs{\nabla\varphi} \leq \abs{\psi'}\abs{\nabla f}\leq \frac{c}{\sqrt{t}} \qquad \text{in }\, \Omega_t;
    \end{equation}
    moreover, since $\Delta f= n\lambda -R \leq n\lambda$ on $M$,
    \begin{equation}\label{phi2}
      \abs{\Delta \varphi}\leq \abs{\psi'\Delta f + \psi''\abs{\nabla f}^2}\leq \frac{c}{t}\qquad \text{in }\, \Omega_t.
    \end{equation}
     Then, integrating by parts we have
\[
\abs{\int_{M} \langle \nabla |\We|, \nabla \varphi \rangle\varphi^2 e^{f}} = \int_{\Omega_{2t}\setminus\Omega_t}\pa{\abs{\Delta\varphi}\varphi^2+2\varphi\abs{\nabla\varphi}^2+\abs{\nabla f}\abs{\nabla\varphi}\varphi^2}|\We| e^{f} .
\]
By the definition of $\We$ and the nonnegative curvature assumption one has $\abs{\We}e^f\leq cR$, and from the previous estimates \eqref{phi1} and \eqref{phi2} we get
\[
\abs{\int_{M} \langle \nabla |\We|, \nabla \varphi \rangle\varphi^2 e^{f}}\leq c\int_{\Omega_{2t}\setminus\Omega_t}R.
\]
By \eqref{stimaeffe} and since $R\in L^1(M^n)$, the left-hand side tends to zero as $t\ra+\infty$, and from \eqref{giani} we obtain, applying Fatou's lemma,
\[
\int_{M} \frac{\sq{Q-(m-2)\lambda|\We|^{2}} \varphi^3 e^{f}}{|\We|}\leq 0.
\]
Now, we use the fact that under our assumptions $Q$ is nonnegative (see Proposition \ref{LE42SBAM} in the Appendix), and since $\lambda$ is strictly negative we get $\We\equiv 0$.  By Bianchi identity we get $R\equiv 0$, and so $g$ is flat by the nonnegative curvature assumption. This concludes the proof of Theorem \ref{TH_expanding}.

%using Kato's inequality and the fact that $|\nabla Ric| \geq c|\nabla R|$ for some positive constant $c$, we get
%\[
%|\nabla |\We|| \leq |\nabla E| \leq \big( |\nabla Ric| + \frac{n}{2}|\nabla R| + |\nabla f||Ric-\frac{1}{2}Rg|\big)e^{-f} \leq c\big(|\nabla Ric|+|\nabla f||Ric|+|\nabla f|R\big)e^{-f} ,
%\]
%for some positive constant $c$. Hence, in $B_{2r}\setminus B_{r}$ we obtain $|\nabla f| \leq c r+d$, for some positive constants $c,d$. In order to estimate $|\nabla|\We||$ it remains to control $|\nabla Ric|$. This follows from $|\nabla Ric|\leq c|\nabla Rm|$, the boundedness of the curvature and Shi's estimates
%\[
%\sup_{B_{2r}}|\nabla Rm| \leq c\sup_{M}|Rm|\Big(\sup_{M}|Rm|+1\Big)^{1/2} \leq c
%\]
%for some positive constant $c$ and every $r>0$, see for instance \cite{bookchow****}, \cite[Lemma 2.6]{deruelle***}. We conclude that in $B_{2r}$ we have
%\[
%|\nabla|\We|| \leq c
%\]

\subsection{The $3$-dimensional case} The proof of Theorem \ref{TH_expanding3dim} in dimension three is formally the same as the higher dimensional case, with some minor corrections. In fact, under the weak assumption of nonnegativity of the Ricci curvature we still have that the full curvature tensor is controlled by the scalar curvature, i.e. $|\operatorname{Rm}|\leq \alpha R$ for some constant $\alpha$. Hence, following the proof in the previous subsection, we obtain that either $(M^{3}, g)$ is flat, or
\[
\int_{M} \frac{\sq{Q-\lambda|\We|^{2}} \varphi^3 e^{f}}{|\We|}\leq 0.
\]
Now, we use the fact that under our assumptions $Q$ is nonnegative (see Proposition \ref{LE42SBAM3D} in the Appendix), and since $\lambda$ is strictly negative we get $\We\equiv 0$. By Bianchi identity we get $R\equiv 0$, and so $g$ is flat by the nonnegative curvature assumption.  This concludes the proof of Theorem \ref{TH_expanding3dim}.

\begin{rem}
  Theorems \ref{TH_expanding} and \ref{TH_expanding3dim} can be proved also using a $L^1$-Liouville property for the operator $\Delta_{-f}$. In fact, since $Q\geq 0$ it follows from equation \eqref{eq_laplTsquared_RiemRwe_mdimT}     that $\Delta_{-f}|\We|\geq 0$ in a distributional sense. Moreover, the nonnegativity of the Ricci curvature implies that $\ricc_{-f} = \ricc -\nabla^2f=2\ricc-\lambda g \geq 0$. Therefore, since $R\in L^1\pa{M^n}$ implies  $|\We|\in  L^1\pa{e^fd\mu, M^n}$, we can apply \cite[Proposition 4.1]{Chambala}, which asserts that on a complete Riemannian manifold $\pa{M^n, g}$ with $\ricc_{-f}\geq 0$ every positive solution $u$ of $\Delta_{-f}u\geq 0$ with $u\in L^1\pa{e^fd\mu, M^n}$ must be constant. Thus $|\We|$ is constant, and therefore zero from equation \eqref{eq_laplTsquared_RiemRwe_mdimT}.
\end{rem}

\section{Steady case: proof of Theorems \ref{TH_steady} and \ref{TH_steady3dim}}\label{SecSteady}

\subsection{The $n$-dimensional case}
Let $\pa{M^n, g}$ be a complete gradient steady Ricci soliton of dimension $n\geq 3$ with nonnegative sectional curvature and assume that
\begin{equation*}
\liminf_{r\ra+\infty}\,\frac{1}{r}\int_{B_{r}} R = 0.
\end{equation*}
In particular, there exists a sequence $\set{r_i}$, $i\in\mathds{N}$, of positive radii converging to $+\infty$ such that
\begin{equation}\label{decayscalar}
  \lim_{i\ra+\infty}\frac{1}{r_{i}}\int_{B_{r_i}} R = 0.
\end{equation}
%  \begin{equation}\label{EQ_expannuli2}
%    \int_{B_{2r}\setminus B_r}R \ra 0 \quad \text{ as }\, r \ra +\infty.
%    \end{equation}
From Proposition \ref{prop_intmain}, we have that either the soliton is Ricci flat (hence flat, since $g$ has nonnegative sectional curvature) or
\begin{equation}\label{giani2}
\int_{M} \frac{Q \,\varphi^3 e^{f}}{|\We|}  \leq - 3\int_{M} \langle \nabla |\We|, \nabla \varphi \rangle\varphi^2 e^{f} \leq 3 \int_{M} |\nabla |\We|||\nabla \varphi| \varphi^2 e^{f}
\end{equation}
for every nonnegative smooth cutoff function $\varphi$ with compact support in $M^n$. Since $(M^n, g)$ has nonnegative sectional curvature,  $|\operatorname{Rm}|\leq \alpha R$, for some positive constant $\alpha$.  Hamilton's identity
\[
R+|\nabla f|^{2} = c
\]
implies that both the scalar curvature $R$ and $|\nabla f|^{2}$ are  bounded. Moreover, it follows e.g. from \cite{CMMRGen} that either $R>0$ or the soliton is Ricci flat, thus flat. So from now on we will assume that the scalar curvature is strictly positive. Using Kato's inequality and the fact that $|\nabla \ricc| \geq |\nabla R|/\sqrt{n}$, we get
\[
|\nabla |\We|| \leq |\nabla E| \leq \big( |\nabla \ricc| + \frac{n}{2}|\nabla R| + |\nabla f||\ricc-\frac{1}{2}Rg|\big)e^{-f} \leq c\big(|\nabla \ricc|+|\nabla f|R\big)e^{-f} \leq c\big(|\nabla \ricc|+ R\big)e^{-f}
\]
for some positive constant $c$. Hence, the left-hand side of \eqref{giani2} can be estimate as
\[
\int_{M} |\nabla |\We|||\nabla \varphi| \varphi^2 e^{f}  \leq \int_{M} |\nabla \ricc| |\nabla \varphi| \varphi^{2}+ \int_{M} R |\nabla \varphi| \varphi^{2} \,.
\]
Now we fix an index $i$ and choose $\varphi$ with support in $B_{2r_i}=B_{2r_i}(o)$ for some origin $o\in M^n$ and such that $\varphi\equiv 1$ in $B_{r_i}$, $|\nabla \varphi| \leq \frac{2}{r_i}$ on $M^n$. Then, by \eqref{decayscalar}, the second term of the left-hand side tends to zero as $i\rightarrow +\infty$. By H\"older inequality and the fact that $R>0$, the remaining term can be estimate as
\begin{equation}\label{estintgradric}
\int_{M} |\nabla \ricc| |\nabla \varphi| \varphi^{2} \leq \left(\int_{M}\frac{|\nabla \ricc|^{2}\varphi^{2}}{R}\right)^{1/2}\left(\int_{M}R|\nabla \varphi|^{2}\varphi^{2}\right)^{1/2}
\end{equation}
To conclude the estimate we need the following lemma.

\begin{lemma}
Let $(M^{n},g)$ be a complete, nonflat, gradient steady Ricci soliton of dimension $n\geq 3$ with nonnegative sectional curvature. Then, for every nonnegative cutoff function $\varphi$ with compact support in $M^n$, there exists a positive constant $c$ such that
\[
\int_{M}\frac{|\nabla \ricc|^{2}\varphi^{2}}{R} \leq c \int_{M} \Big( R + R|\nabla \varphi|^{2}\Big) \varphi^{2} .
\]
\end{lemma}
\begin{proof}
First of all, in some local frame, we have
\[
|\nabla \ricc|^{2} = \frac{1}{2} |\nabla_{k}R_{ij}-\nabla_{j}R_{ik}|^{2} + \nabla_{k}R_{ij}\nabla_{j}R_{ik}\,.
\]
From the soliton equation and the commutation rule of covariant derivatives, one has
\[
\nabla_{k}R_{ij}-\nabla_{j}R_{ik} = R_{kijl} \nabla_{l} f\,.
\]
Since $|\operatorname{Rm}|\leq \alpha R$ and $|\nabla f|^{2}\leq c$ for some $\alpha, c>0$, we obtain
\begin{equation}\label{estgradric}
|\nabla \ricc|^{2} \leq c R^{2} + \nabla_{k}R_{ij}\nabla_{j}R_{ik}.
\end{equation}
Hence, to finish the proof we have to estimate the right-hand side. Integrating by parts, commuting indices and using Young's inequality we get
\begin{align*}
\int_{M} \frac{\nabla_{k}R_{ij}\nabla_{j}R_{ik} \varphi^{2}}{R} &= - \int_{M}\frac{R_{ij}\nabla_{k}\nabla_{j}R_{ik} \varphi^{2}}{R} -2\int_{M}\frac{R_{ij}\nabla_{j}R_{ik} \nabla_{k}\varphi \varphi}{R} + \int_{M}\frac{R_{ij}\nabla_{j}R_{ik} \nabla_{k}R \varphi^{2}}{R^{2}} \\
&\leq - \int_{M}\frac{R_{ij}\nabla_{j}\nabla_{k}R_{ik} \varphi^{2}}{R} - \int_{M} \frac{(R_{kjil}R_{ij}R_{kl}+R_{ij}R_{ik}R_{jl})\varphi^{2}}{R} \\
&+ \eps \int_{M}\frac{|\nabla \ricc|^{2}\varphi^{2}}{R} + c(\eps)\int_{M} \Big( R |\nabla \varphi|^{2} + |\nabla R|^{2} \Big) \varphi^{2} ,
\end{align*}
for every $\eps>0$ and some constant $c(\eps)$. Using Bianchi identity, the fact that $|\operatorname{Rm}|\leq \alpha R$ and the well known soliton identity $\nabla R = 2 \ricc(\nabla f)$ (see e.g. \cite{ELNM} or \cite{MRR_GS}), we obtain
\begin{align*}
\int_{M} \frac{\nabla_{k}R_{ij}\nabla_{j}R_{ik} \varphi^{2}}{R} &\leq -\frac{1}{2} \int_{M}\frac{R_{ij}\nabla_{i}\nabla_{j}R \varphi^{2}}{R} + \eps \int_{M}\frac{|\nabla \ricc|^{2}\varphi^{2}}{R} + c(\eps)\int_{M} \Big( R + R |\nabla \varphi|^{2}\Big) \varphi^{2} \\
&= \frac{1}{4} \int_{M}\frac{|\nabla R|^{2} \varphi^{2}}{R} + \eps \int_{M}\frac{|\nabla \ricc|^{2}\varphi^{2}}{R} + c_{1}(\eps)\int_{M} \Big( R + R |\nabla \varphi|^{2}\Big) \varphi^{2} \\
& \leq \eps \int_{M}\frac{|\nabla \ricc|^{2}\varphi^{2}}{R} + c_{2}(\eps)\int_{M} \Big( R + R |\nabla \varphi|^{2}\Big) \varphi^{2} ,
\end{align*}
for every $\eps>0$ and some constant $c_{2}(\eps)$. Choosing $\eps \ll 1$, this estimate and \eqref{estgradric} conclude the proof of the lemma.
\end{proof}
Now we can return to the proof of Theorem \ref{TH_steady}. Using the previous lemma and \eqref{estintgradric}, we obtain
\[
\int_{M} |\nabla \ricc| |\nabla \varphi| \varphi^{2} \leq c\left(  \int_{M} \Big( R + R|\nabla \varphi|^{2}\Big) \varphi^{2}\right)^{1/2} \left(\int_{M}R|\nabla \varphi|^{2}\varphi^{2}\right)^{1/2} \leq \frac{c}{r_i} \int_{B_{2r_i}} R
\]
which, by \eqref{decayscalar}, tends to zero as $i\ra +\infty$. Applying Fatou's lemma, from \eqref{giani2}, we get
\[
\int_{M}  \frac{Q \,\varphi^3 e^{f}}{|\We|} \leq 0.
\]
Hence, Proposition \ref{LE42SBAM} implies that $Q \equiv 0$ on $M$. The equality case implies that the Ricci tensor at every point has at most two distinct eigenvalues $\Lambda=0$ with multiplicity $(n-2)$ and $\Upsilon=\frac12 R$ with multiplicity two. In order to conclude the proof we need the following general result (see Lemma 3.2 in \cite{PetWy}) for constant rank, symmetric, nonnegative tensors. We recall the definition of the $X$-Laplacian $\Delta_X=\Delta - g\pa{X, \cdot}$, for some smooth vector field $X$ (see e.g. \cite{CMMRGen}).
\begin{lemma}\label{LEPet}
  Let $\mathcal{T}$ be a constant rank, symmetric, nonnegative tensor on some tensor bundle. If $\pa{\Delta_X\mathcal{T}}(V, V)\leq 0$ for $V\in \operatorname{Ker}\mathcal{T}$ and $X$ is a vector field, then the kernel of $\mathcal{T}$ is a parallel subbundle.
\end{lemma}

Now let $e_1, \ldots, e_n$ be a local orthonormal frame such that $\ricc\pa{e_1, e_1}=\ricc\pa{e_2, e_2}=\frac12 R$ and $\ricc\pa{e_k, \cdot}=0$ for every $k=3, \ldots, n$. In particular, the only nonzero sectional curvature is $\sigma_{12}$, where $\sigma_{ij}$ is the sectional curvature defined by the two-plane spanned by $e_{i}$ and $e_{j}$.
Then, by the well known soliton identity (see e.g. \cite{ELNM} or \cite{MRR_GS})
\[
\Delta_{\nabla f}R_{ij} = 2\lambda R_{ij}-2R_{iljt}R_{lt}= -2R_{iljt}R_{lt},
\]
we see that, for every fixed $k=3, \ldots, n$, $e_k\in \operatorname{Ker}\ricc$ and
\[
\pa{\Delta_{\nabla f}\ricc}\pa{e_k, e_k} = -2 \sum_i\operatorname{Rm}\pa{e_k, e_i, e_k, e_i}\ricc\pa{e_i, e_i} = -R\pa{\sigma_{1k}+\sigma_{2k}}= 0.
\]
Moreover,  since $g$ has nonnegative sectional curvature, $\frac12 Rg-\ricc$ is nonnegative, and a simple computation shows that
\[
\pa{\Delta_{\nabla f}\pa{\frac12 Rg-\ricc}}\pa{e_1, e_1} = \pa{\Delta_{\nabla f}\pa{\frac12 Rg-\ricc}}\pa{e_2, e_2} = 0.
\]
 Now  Lemma \ref{LEPet} applies, and by de Rham Decomposition Theorem (see for instance \cite{KobNom1}, Chapter 1, Section 6) the metric splits and Theorem \ref{TH_steady} follows since, as we said in the Introduction, the cigar soliton $\Sigma^2$ is the only complete two-dimensional steady soliton with positive curvature.
 \begin{rem}
   Note that, under our assumptions, $f$ in general is not proper, thus one cannot exploit the argument used in the expanding case involving a cutoff function depending on the potential $f$.
 \end{rem}

\subsection{The $3$-dimensional case}

The proof of Theorem \ref{TH_steady3dim} in dimension three follows the lines of the higher dimensional case. First of all, by Chen \cite{chen} we have that $g$ must have nonnegative sectional curvature, and since $R+\abs{\nabla f}^2=c$, $g$ has also bounded curvature. From Hamilton's strong maximum principle (see e.g. \cite{ChowLuNi}) we deduce that $\pa{M^3, g}$ is either flat, or it splits as a product $\erre \times \Sigma^2$ (where $\Sigma^2$ is  again the cigar steady soliton) or it has strictly positive sectional curvature. In the latter case, following the proof of  Theorem \ref{TH_steady} we obtain
\[
\int_{M} \frac{Q \varphi^3 e^{f}}{|\We|}\leq 0.
\]
Now, we use the fact that under our assumptions $Q$ is nonnegative (see Proposition \ref{LE42SBAM3D} in the Appendix), and  we get $Q\equiv 0$. The equality case in Proposition \ref{LE42SBAM3D} implies that the Ricci curvature has a zero eigenvalue, a contradiction. This concludes the proof of Theorem \ref{TH_expanding3dim}.

\subsection{Proof of Corollary \ref{corquadratic}}
Corollary \ref{corquadratic} is a direct consequence of the following
\begin{lemma}[Lemma 4.3 in \cite{deruelle}]\label{lemma4.3deruelle}
  Let $\pa{M^n, g}$ be a complete gradient steady Ricci soliton with nonnegative Ricci curvature. Then, for every $o \in M^n$ and every $r>0$,
  \[
  \int_{B_r(o)} R \leq n\sqrt{c}\frac{\operatorname{Vol}\pa{B_r(o)}}{r},
    \]
    where $c$ is the constant in Hamilton's identity $R+\abs{\nabla f}^2=c$.
\end{lemma}
\begin{proof}
  Integrating the equation $R+\Delta f =0$ one has
  \begin{align*}
    \int_{B_r(o)} R &= - \int_{B_r(o)} \Delta f \leq \int_{\partial B_r(o)}\abs{\nabla f} \leq C A\pa{\partial B_r(o)},
  \end{align*}
  where $A\pa{\partial B_r(o)}$ is the $(n-1)$-dimensional volume of the geodesic sphere $\partial B_r(o)$. Now, since $\pa{M^n, g}$ has nonnegative Ricci curvature, by Bishop-Gromov theorem (see for instance \cite{ChowLuNi}) for every $o\in M$ and every $r>0$ one has
  \[
  \frac{rA\pa{\partial B_r(o)}}{\operatorname{Vol}\pa{B_r(o)}} \leq n,
  \]
  which implies the thesis.

\end{proof}

\section{Appendix}
We provide here the proof of some new algebraic curvature estimates needed for the proof of the main theorems. First we recall the following lemma (see \cite[Proposition 3.1]{CatinoPre})

\begin{lemma}\label{p-est}
Let $(M^{n},g)$ be a Riemannian manifold of dimension $n\geq 3$ with nonnegative sectional curvature. Then, the following estimate holds
\[
R_{ikjl} T_{ij} T_{kl} \leq \frac{n-2}{2n} R|T|^{2} ,
\]
with equality if and only if the Ricci tensor has at most two distinct eigenvalues, $\Lambda$ of multiplicity $(n-2)$ and $\Upsilon$ of multiplicity two.
\end{lemma}
\begin{proof} Let $\{e_{i}\}$, $i=1,\ldots, n$, be the eigenvectors of $T$ and let $\lambda_{i}$ be the corresponding eigenvalues. Moreover, let $\sigma_{ij}$ be the sectional curvature defined by the two-plane spanned by $e_{i}$ and $e_{j}$. We want to prove that the quantity
\[
R_{ikjl} T_{ij} T_{kl}  - \frac{n-2}{2n} R|T|^{2} \,=\, \sum_{i,j=1}^{n} \lambda_{i}\lambda_{j} \sigma_{ij} - \frac{n-2}{2n} R \sum_{k=1}^{n} \lambda_{k}^{2}
\]
is nonpositive if $\sigma_{ij} \geq 0$ for all $i,j=1,\ldots,n$. The scalar curvature can be written as
\[
R \,=\, R_{ijij} \,=\, \sum_{i,j=1}^{n} \sigma_{ij} \,=\, 2 \sum_{i<j}\sigma_{ij}\,.
\]
Hence, one has the following
\begin{eqnarray*}
\sum_{i,j=1}^{n} \lambda_{i}\lambda_{j} \sigma_{ij} - \frac{n-2}{2n}  R \sum_{k=1}^{n} \lambda_{k}^{2} &=& 2 \sum_{i<j}\lambda_{i}\lambda_{j} \sigma_{ij} - \frac{n-2}{n}  \sum_{i<j} \sigma_{ij} \sum_{k=1}^{n} \lambda_{k}^{2} \\
&=& \sum_{i<j} \Big(2\lambda_{i}\lambda_{j} -\frac{n-2}{n}\sum_{k=1}^{n} \lambda_{k}^{2} \Big) \sigma_{ij}\,.
\end{eqnarray*}
On the other hand, one has
\[
\sum_{k=1}^{n} \lambda_{k}^{2} \,=\, \lambda_{i}^{2} + \lambda_{j}^{2} + \sum_{k\neq i,j}\lambda_{k}^{2}\,.
\]
Moreover, using the Cauchy-Schwarz inequality and the fact that $\sum_{k=1}^{n} \lambda_{k} =0$, we obtain
\[
\sum_{k\neq i,j}\lambda_{k}^{2} \,\geq \, \frac{1}{n-2} \Big( \sum_{k\neq i,j}\lambda_{k}\Big)^{2} \,=\, \frac{1}{n-2} \big( \lambda_{i} + \lambda_{j} \big)^{2} \,,
\]
with equality if and only if $\lambda_{k}=\lambda_{k'}$ for every $k,k'\neq i,j$. Hence, the following estimate holds
$$
\sum_{k=1}^{n} \lambda_{k}^{2} \,\geq \, \frac{n-1}{n-2} \big(\lambda_{i}^{2} + \lambda_{j}^{2} \big) + \frac{2}{n-2} \lambda_{i} \lambda_{j}\,.
$$
Using this, since $\sigma_{ij} \geq 0$, it follows that
\begin{eqnarray*}
\sum_{i,j=1}^{n} \lambda_{i}\lambda_{j} \sigma_{ij} - \frac{n-2}{2n} R \sum_{k=1}^{n} \lambda_{k}^{2} &\leq& \frac{n-1}{n}\sum_{i<j} \big(2\lambda_{i}\lambda_{j} -\big(\lambda_{i}^{2} + \lambda_{j}^{2} \big) \Big) \sigma_{ij} \\
&=&- \frac{n-1}{n}\sum_{i<j} (\lambda_{i}-\lambda_{j})^{2} \sigma_{ij}\, \leq \, 0 \,.
\end{eqnarray*}
This concludes the proof of the estimate.

In the equality case we have that $\lambda_{k}=\lambda_{k'}$ for every $k,k'\neq i,j$ and for $i<j$
$$
(\lambda_{i}-\lambda_{j})^{2}\sigma_{ij} \,=\, 0\,.
$$
Hence, either, for every $i<j$, $\sigma_{ij}=0$ (and $\operatorname{Rm}=0$) or there exists $i<j$ such that $\sigma_{ij}>0$. In this second case, without loss of generality we can assume that $\sigma_{12}>0$ and we have that $\lambda_{1}=\lambda_{2}$, $\lambda_{3}=\ldots=\lambda_{n}$. Since $\ricc= T+ \frac{1}{n}Rg$, the conclusion on the Ricci tensor follows and this concludes the proof of the lemma.
\end{proof}

\begin{proposition}\label{LE42SBAM}
 Let $\pa{M^n, g}$ be a Riemannian manifold of dimension $n\geq 3$ with nonnegative sectional curvature. Then
 \[
 P:=\frac{(n-2)^3}{4n^2}R^3-2R_{ikjt}T_{ij}T_{kt}-\frac{\pa{n-2}\pa{n-4}}{2n}R\abs{T}^2\geq 0,
 \]
 with equality  if and only if the Ricci tensor has at most two distinct  eigenvalues, $\Lambda=0$ with multiplicity $(n-2)$ and $\Upsilon=\frac{1}{2}R$ with multiplicity two.
\end{proposition}
\begin{proof}
  From Lemma \ref{p-est} we deduce that
  \[
  -2R_{ikjl} T_{ij} T_{kl} \geq -\frac{n-2}{n} R|T|^{2};
  \]
  using the previous estimate in the definition of $P$ we get
  \begin{equation}\label{estSbam}
  P \geq \frac{\pa{n-2}^3}{4n^2}R^3-\frac{\pa{n-2}^2}{2n}R\abs{T}^2=\frac{\pa{n-2}^2}{2n}R\sq{\frac{\pa{n-2}}{2n}R^2-\abs{T}^2}.
  \end{equation}
  By the nonnegativity assumption on the sectional curvature we know that
  \[
  \abs{\ricc}^2\leq \frac{1}{2}R^2,
  \]
  which implies
  \[
  \abs{T}^2 \leq \frac{n-2}{2n}R^2.
  \]
  Inserting the previous relation in \eqref{estSbam} we get $P\geq 0$.

If $P=0$ at a point, we have $|\ricc|^{2}=\frac{1}{2}R^{2}$ and the equality case in Lemma \ref{p-est}. Hence, the Ricci tensor has at most two distinct eigenvalues, $\Lambda$ of multiplicity $(n-2)$ and $\Upsilon$ of multiplicity two. In particular $R = (n-2)\Lambda + 2 \Upsilon$. Combining this with the identity
\[
|\ricc|^{2}=(n-2)\Lambda^{2}+2\Upsilon^{2} = \frac{1}{2}R^{2}
\]
we obtain
\[
\Lambda^{2} = \frac{2}{n} R \Lambda .
\]
Now either $\Lambda=0$ and $\Upsilon=\frac{1}{2}R$ or $\Lambda =\frac{2}{n}R$ and $\Upsilon = -\frac{n-2}{2n} R$. But, since $g$ has nonnegative sectional curvature, this second case implies $R=0$ and so $g$ is flat. In both cases we have the splitting result and this concludes the proof of the proposition.
\end{proof}
For the three dimensional case we  need the following  algebraic lemma.
\begin{lemma}\label{dario}
For $x,y,z\geq0$ let
$$P(x,y,z)=5(x^3+y^3+z^3)-5(x^2y+xy^2+x^2z+xz^2+y^2z+yz^2)+18xyz.$$
Then $P(x,y,z)\geq3xyz\geq0$ and $P(x,y,z)=0$ if and only if $x=0$
and $y=z$, or $y=0$ and $x=z$, or $z=0$ and $x=y$.
\end{lemma}

\begin{proof} It is easy to see that
\begin{equation}\label{P1}
P(x,y,z)=5x(x-z)(x-y)+5y(y-z)(y-x)+5z(z-x)(z-y)+3xyz.
\end{equation}
Since $P(x,y,z)$ is symmetric in $(x,y,z)$, i.e. it is invariant
under any permutation of the variables $x,y,z$, we can assume
without loss of generality that $0\leq x\leq y\leq z$. Hence
\begin{equation}\label{P2}
\begin{array}{rcl}
\bar{P}(x,y,z)&:=&5x(x-z)(x-y)+5y(y-z)(y-x)+5z(z-x)(z-y)\\
&=&5x(x-z)(x-y)+5(z-y)^2(z+y-x)\geq0.
\end{array}
\end{equation}
From \eqref{P1} and \eqref{P2} we conclude that for every
$x,y,z\geq0$
\begin{equation}\label{P3}
P(x,y,z)=\bar{P}(x,y,z)+3xyz\geq3xyz\geq0.
\end{equation}
If $P(x,y,z)=0$, by \eqref{P3} we have that $xyz=0$. If $x=0$ then
$$0=P(0,y,z)=5(y^3+z^3-y^2z-yz^2)=5(y-z)^2(y+z),$$ thus we have
$y=z$. The cases when $y=0$ or $z=0$ can be obtained by permutation
of the variables $x,y,z$.
\end{proof}

\begin{proposition}\label{LE42SBAM3D}
 Let $\pa{M^3, g}$ be a three dimensional Riemannian manifold with nonnegative Ricci curvature. Then
 \[
 P:= 4R_{ij}R_{jk}R_{ki}-\frac{7}{2}R\abs{\ricc}^2+\frac{3}{4}R^3\geq 0,
 \]
 with equality   if and only if the Ricci tensor has at most two distinct  eigenvalues, $\Lambda=0$ and $\Upsilon=\frac{1}{2}R$ with multiplicity two.
\end{proposition}
\begin{proof}
  Let $e_1, e_2, e_3$ be a local orthonormal frame such that $\ricc(e_i, \cdot)=\mu_ie_i$ for $i=1, 2, 3$. Then
  \begin{align*}
  4P&=\pa{\mu_1^3+\mu_2^3+\mu_3^3}-14\pa{\mu_1+\mu_2+\mu_3}\pa{\mu_1^2+\mu_2^2+\mu_3^2}+3\pa{\mu_1+\mu_2+\mu_3}^3 \\ &= 5(\mu_1^3+\mu_2^3+\mu_3^3)-5(\mu_1^2\mu_2+\mu_1\mu_2^2+\mu_1^2\mu_3+\mu_1\mu_3^2+\mu_2^2\mu_3+\mu_2\mu_3^2)+18\mu_1\mu_2\mu_3.
  \end{align*}
  Now the proposition follows from Lemma \ref{dario}.
\end{proof}

\begin{ackn}
 The second author would like to thank R. Bertani,  R. Locatelli and M. Porzio for some useful remarks and suggestions.
\end{ackn}

\bibliographystyle{plain}

\bibliography{BiblioCodazzism}
\end{document}